\documentclass{article}
\usepackage[top=1in,bottom=1.1in,left=1in,right=1in]{geometry}
\usepackage{amssymb}
\usepackage{amsmath}
\usepackage{amsthm}
\usepackage{verbatim}
\usepackage{calrsfs}
\DeclareMathAlphabet{\oldcal}{OMS}{zplm}{m}{n}

\newtheorem{theorem}{Theorem}[section]

\newtheorem{corollary}[theorem]{Corollary}

\newtheorem{remark}[theorem]{Remark}
\newtheorem{lemma}[theorem]{Lemma}
\newtheorem{definition}[theorem]{Definition}

\begin{document}

\title{Continuous Selections of Lipschitz Extensions in Metric Spaces}
\author{Rafa Esp\'{i}nola$^{a}$, Adriana Nicolae$^{b,c}$}
\date{}

\maketitle
\begin{center}
{\footnotesize
$^{a}$Departamento de An\'alisis Matem\'atico - IMUS, Universidad de Sevilla, Apdo. de Correos 1160, 41080 Sevilla, Spain
\ \\
$^{b}$Department of Mathematics, Babe\c s-Bolyai University, Kog\u alniceanu 1, 400084 Cluj-Napoca, Romania
\ \\
$^{c}$Simion Stoilow Institute of Mathematics of the Romanian Academy, Research group of the project PD-3-0152,\\ 
P.O. Box 1-764, RO-014700 Bucharest, Romania\\
\ \\
E-mail addresses:  espinola@us.es (R. Esp\' inola), anicolae@math.ubbcluj.ro (A. Nicolae)
}
\end{center}

\begin{abstract}
This paper deals with the study of parameter dependence of extensions of Lipschitz mappings from the point of view of continuity. We show that if assuming appropriate curvature bounds for the spaces, the multivalued extension operators that assign to every nonexpansive (resp. Lipschitz) mapping all its nonexpansive extensions (resp. Lipschitz extensions with the same Lipschitz constant) are lower semi-continuous and admit continuous selections. Moreover, we prove that Lipschitz mappings can be extended continuously even when imposing the condition that the image of the extension belongs to the closure of the convex hull of the image of the original mapping. When the target space is hyperconvex one can obtain in fact nonexpansivity.\\


\end{abstract}

\section{Introduction}
The Kirszbraun theorem \cite{Kir34} is a fundamental result in the theory of Lipschitz extensions and states that for any Lipschitz function $f : A \subseteq \mathbb{R}^n \to \mathbb{R}^m$ there exists a Lipschitz extension $f' : \mathbb{R}^n \to \mathbb{R}^m$ with the same Lipschitz constant. The result for arbitrary Hilbert spaces goes back to Valentine \cite{Val45}. Aronszajn and Panitchpakdi \cite{AroPan56} introduced the concept of hyperconvexity, which is closely related to this problem since a metric space $Y$ is hyperconvex if and only if given any subspace $A$ of any metric space $X$, every nonexpansive mapping $f : A \to Y$ admits a nonexpansive extension to $X$.

The first result that extends Kirszbraun's theorem to the metric setting by imposing curvature bounds in the sense of Alexandrov was given by Lang and Schroeder in \cite{LanSch97} (see also \cite{KucSta88,Val45} for previous particular results). The same problem was later approached by Alexander, Kapovitch and Petrunin in \cite{AleKapPet11} where a different proof method is considered.

All the aforementioned results guarantee the existence of an extension for the original mapping. However, this extension is not necessary unique and no information is given on the parameter dependence of the extensions. Kopeck\'{a} studied the process of assigning extensions to mappings from the point of view of continuity providing positive answers first in Euclidean \cite{Kop12a} and then in Hilbert spaces \cite{Kop12}. Namely, the multivalued extension mappings that assign to every nonexpansive (resp. Lipschitz) mapping all its nonexpansive extensions (resp. Lipschitz extensions with the same Lipschitz constant) are proved to be lower semi-continuous using Kirszbraun's theorem and a homotopy argument. Applying Michael's selection theorem one obtains continuous selections of these multivalued extension operators. Kopeck\'{a} and Reich further generalized these results in \cite{KopRei11}, obtaining a continuous singlevalued extension operator with the additional condition that the image of the extension belongs to the closed convex hull of the image of the original mapping.

A natural question is to study this problem in geodesic metric spaces with curvature bounds in the sense of Alexandrov since in this context a generalized version of Kirszbraun's theorem holds. Here we show that one can indeed prove counterparts of such continuity results in this setting. In Section \ref{sect-lsc-cont-sel} we show that assuming appropriate curvature bounds for the spaces, the multivalued extension mappings are lower semi-continuous and admit continuous selections. Moreover, we prove in Section \ref{convexity} that Lipschitz mappings can be extended continuously even when imposing the above mentioned convexity condition on the image of the extension. Section \ref{hyperconvexity} deals with the case where the target space is hyperconvex and shows that in this situation one can obtain in fact nonexpansivity.

\section{Preliminaries} \label{prelim}
Let $(X,d)$ be a metric space. A {\it geodesic path} from $x$ to $y$ is a mapping $c:[0,l] \subseteq \mathbb{R} \to X$ such that $c(0) = x, c(l) = y$ and $d\left(c(t),c(t^{\prime})\right) = \left|t - t^{\prime}\right|$ for every $t,t^{\prime} \in [0,l]$. The image $c\left([0,l]\right)$ of $c$ forms a {\it geodesic segment} which joins $x$ and $y$. Note that a geodesic segment from $x$ to $y$ is not necessarily unique. $(X,d)$ is a {\it geodesic space} if every two points in $X$ can be joined by a geodesic path. A point $z\in X$ belongs to a geodesic segment joining $x$ and $y$ if and only if there exists $t\in [0,1]$ such that $d(z,x)= td(x,y)$ and $d(z,y)=(1-t)d(x,y)$, and we will write $z=(1-t)x+ty$ for simplicity. For more details on geodesic metric spaces the reader may check \cite{Bri99}.

A geodesic space $(X,d)$ is {\it Busemann convex} if given any pair of geodesic paths $c_1 : [0, l_1] \to X$ and $c_2 : [0,l_2] \to X$ with $c_1(0) = c_2(0)$ one has 
\[d(c_1(tl_1),c_2(tl_2)) \le td(c_1(l_1),c_2(l_2)), \quad \text{for every } t \in [0,1].\] 

A subset $C$ of $X$ is {\it convex} if any geodesic segment that joins every two points of $C$ is contained in $C$. Let $G_1(C)$ denote the union of all geodesics segments with endpoints in $C$. Note that $C$ is convex if and only if $G_1(C) = C$. Recursively, for $n \ge 2$ we set $G_n(C) = G_1(G_{n-1}(C))$. The {\it convex hull} of $C$ is
\[\mbox{co}(C) = \bigcup_{n \in \mathbb{N}}G_n(C).\]
By $\overline{\mbox{co}}(C)$ we denote the closure of the convex hull. It is easy to see that in a Busemann convex geodesic space, the closure of the convex hull is convex and hence it is the smallest closed convex set containing $C$.

For $\kappa \in \mathbb{R}$ let $M^2_{\kappa}$ denote the complete, simply connected model surface of constant curvature $\kappa$. In the sequel we assume that $\kappa \le 0$.

A {\it geodesic triangle} $\Delta = \Delta(x_1,x_2,x_3)$ consists of three points $x_1, x_2$ and $x_3$ in $X$ and three geodesic segments corresponding to each pair of points. A {\it $\kappa$-comparison triangle} for $\Delta$ is a triangle $\bar{\Delta} = \Delta(\bar{x}_1, \bar{x}_2, \bar{x}_3)$ in $M^2_{\kappa}$ such that $d(x_i,x_j) = d_{M^2_{\kappa}}(\bar{x}_i,\bar{x}_j)$ for $i,j \in \{1,2,3\}$. For $\kappa$ fixed, $\kappa$-comparison triangles of geodesic triangles always exist and are unique up to isometry. 

A geodesic triangle $\Delta$ satisfies the {\it CAT$(\kappa)$} (resp. {\it reversed CAT$(\kappa)$}) {\it inequality} if for every $\kappa$-comparison triangle $\bar{\Delta}$ of $\Delta$ and for every $x,y \in \Delta$ we have
\[d(x,y) \le d_{M^2_{\kappa}}(\bar{x},\bar{y}) \mbox{ (resp. } d(x,y) \ge d_{M^2_{\kappa}}(\bar{x},\bar{y})\mbox{)},\]
where  $\bar{x},\bar{y} \in \bar{\Delta}$ are the corresponding points of $x$ and $y$, i.e., if $x = (1-t)x_i + tx_j$ then $\bar{x} = (1-t)\bar{x}_i + t\bar{x}_j$.

A {\it CAT$(\kappa)$ space} (also known as a space of curvature bounded above by $\kappa$ in the sense of Alexandrov) is a geodesic space for which every geodesic triangle satisfies the CAT$(\kappa)$ inequality. Any CAT$(0)$ space (and so any CAT$(\kappa)$ space) is Busemann convex.

A geodesic metric space is said to have  curvature bounded below by $\kappa$ in the sense of Alexandrov (denoted by {\it CBB$(\kappa)$}) if every geodesic triangle satisfies the reversed CAT$(\kappa)$ inequality. If $X$ is a CBB$(\kappa)$ space, then the direct product $X \times M_\kappa^2$ is a CBB$(\kappa)$ space with the metric
\begin{equation} \label{metric-direct-product}
d\left((x,a),(y,b)\right)^2 = d_X(x,y)^2 + d_{M_\kappa^2}(a,b)^2.
\end{equation}
Other properties of spaces with curvature bounded above or below and equivalent definitions can be found in \cite{Bri99,Bur01}.

Let $(X,d)$ be a metric space. Taking $x \in X$ and $r > 0$ we denote the closed ball centered at $x$ with radius $r$ by $B(z,r).$ Given $C$ a nonempty subset of $X$, the {\it distance of a point} $x \in X$ to $C$ is $\mbox{dist}(x,C) = \inf\{d(x,c) : c \in C\}.$ If $B$ and $C$ are nonempty subsets of $X$, one defines the {\it Pompeiu-Hausdorff distance} as
\[H(B,C) = \max\left\{\sup_{b \in B}\text{dist}(b,C), \sup_{c \in C}\text{dist}(c,B)\right\}.\]
The {\it metric projection} $P_C$ onto $C$ is the mapping
\[P_C(x)=\{ c \in C : d(x,c)=\mbox{dist}(x,C)\}, \quad \text{for every } x\in X.\]
In any CAT$(0)$ space the metric projection onto a convex and complete subset is a singlevalued and nonexpansive (that is, $1$-Lipschitz) mapping.

A metric space $X$ is {\it hyperconvex} if $\bigcap_{\alpha}B(x_\alpha,r_\alpha) \ne \emptyset$ for every collection of points $\{x_\alpha\}$ in $X$ and positive numbers $\{r_\alpha\}$ such that $d(x_\alpha,x_\beta) \le r_\alpha + r_\beta$ for any $\alpha, \beta$. A subset $E$ of a metric space $X$ is called {\it externally hyperconvex} (with respect to $X$) if given any family $\{x_{\alpha}\}$ of points in $X$ and any family $\{r_{\alpha}\}$ of real numbers satisfying
\[d(x_{\alpha},x_{\beta}) \le r_{\alpha} + r_{\beta} \quad \mbox{and} \quad \mbox{dist}(x_\alpha, E) \le r_{\alpha},\] 
it follows that $\bigcap_{\alpha}B(x_{\alpha},r_{\alpha}) \cap E \ne \emptyset$. For a more detailed discussion on hyperconvex metric spaces, see \cite{EspKha01}.

Let $(X,d_X)$, $(Y,d_Y)$ be metric spaces, $A \subseteq X$ nonempty and consider $C(A,Y)$ the family of bounded and continuous mappings from $A$ to $Y$. For each $f,g \in C(A,Y)$, let $d_\infty(f,g)=\sup_{x \in A}d_Y(f(x),g(x))$. Endowed with the supremum distance $d_\infty$, $C(A,Y)$ is a metric space which is complete if $Y$ is complete. We consider two subsets of $C(A,Y)$: $\mathcal{L}(A,Y)$ which includes all bounded Lipschitz mappings from $A$ to $Y$ and is not necessarily a closed subset of $C(A,Y)$ and $\mathcal{N}(A,Y)$ which stands for the family of all bounded nonexpansive mappings defined from $A$ to $Y$ and which is closed in $C(A,Y)$.

For $f \in \mathcal{L}(A,Y)$ we denote the {\it smallest Lipschitz constant} of $f$ on $B \subseteq A$ by $\text{Lip}(f, B)$. More precisely,
\[\text{Lip}(f,B) = \sup\left\{\frac{d_Y(f(x),f(y))}{d_X(x,y)} : x,y \in B, x \ne y\right\}.\]

For a set $C$, we denote by $\oldcal{P}(C)$ the family of all its subsets. We consider two multivalued extension mappings:
\begin{itemize}
\item $\Phi : \mathcal{N}(A,Y) \to \oldcal{P}\left(\mathcal{N}(X,Y)\right)$ which assigns to each nonexpansive mapping $f \in \mathcal{N}(A,Y)$ all its nonexpansive extensions $f' \in \mathcal{N}(X,Y)$. Note that in this case it may happen that $\text{Lip}(f,A) < \text{Lip}(f',X) \le 1$. 
\item $\Psi : \mathcal{L}(A,Y) \to \oldcal{P}\left(\mathcal{L}(X,Y)\right)$ which assigns to each Lipschitz mapping $f \in \mathcal{L}(A,Y)$ all its Lipschitz extensions $f' \in \mathcal{L}(X,Y)$ with $\text{Lip}(f,A) = \text{Lip}(f',X)$.
\end{itemize}

Recall that having two topological spaces $X$ and $Y$, a multivalued mapping $\Gamma : X \to \oldcal{P}(Y)$ is {\it lower semi-continuous} if for every open $V \subseteq Y$, the set $\{x \in X : \Gamma(x) \cap V \ne \emptyset\}$ is open in $X$. If $X$ and $Y$ are metric spaces, $\Gamma$ is {\it nonexpansive} if $H(\Gamma(x),\Gamma(y)) \le d_X(x,y)$ for every $x,y \in X$.

The classical Kirszbraun theorem was extended to geodesic metric spaces with lower and upper curvature bounds by Lang and Schroeder in \cite{LanSch97}. Later, Alexander, Kapovitch and Petrunin considered a different approach of the proof in \cite{AleKapPet11}.

\begin{theorem}[Lang, Schroeder \cite{LanSch97}] \label{thm-gen-Kirszbraun}
Let $\kappa \le 0$, $X$ a CBB$(\kappa)$ space and $Y$ a complete CAT$(\kappa)$ space. Suppose $A \subseteq X$ is nonempty and $f:A\to Y$ is nonexpansive. Then there exists a nonexpansive extension $f':X \to Y$ of $f$. 
\end{theorem}

Although the result can be also stated when $\kappa > 0$ with an appropriate boundedness condition on the set $f(A)$, here we are only concerned with the case $\kappa \le 0$.

For $\kappa = 0$ the result can be generalized to any arbitrary Lipschitz constant by scaling the metric on either $X$ or $Y$ and so we may consider both mappings $\Phi$ and $\Psi$. When $\kappa < 0$, the same argument can be applied for Lipschitz constants greater than $1$. However, for Lipschitz constants strictly less than $1$, we cannot expect the result to hold true. Suppose one could extend all mappings $f : A \subseteq \mathbb{H}^2 \to \mathbb{H}^2$ with $\text{Lip}(f,A) < 1$ while keeping the same Lipschitz constant. Taking $\kappa \in (-1,0)$, this implies that we can extend all nonexpansive mappings defined on $A \subseteq \mathbb{H}^2$ with values in $M_\kappa^2$ to nonexpansive mappings on $\mathbb{H}^2.$ But this means that $M_\kappa^2$ is a CAT$(-1)$ space (see Proposition 6.2 in \cite{LanSch97}), a contradiction. Since in this work we rely on Theorem \ref{thm-gen-Kirszbraun} in order to obtain our continuity results, for the case $\kappa < 0$ we will only study the mapping $\Phi$.

However, if the target space is an $\mathbb{R}$-tree, then it was proved in \cite{LanSch97} that we not only can extend mappings with arbitrary Lipschitz constant, but we can also drop the curvature assumption on the source space. 

\begin{theorem}[Lang, Schroeder \cite{LanSch97}] \label{thm-R-trees}
Let $X$ be a metric space and $Y$ a complete $\mathbb{R}$-tree. Suppose $A \subseteq X$ is nonempty and $f:A\to Y$ is a Lipschitz mapping. Then there exists a Lipschitz extension $f':X \to Y$ of $f$ with $\emph{Lip}(f',X) = \emph{Lip}(f,A)$.
\end{theorem} 

Theorem \ref{thm-R-trees} is a consequence of the following extension theorem proved for hyperconvex metric spaces by Aronszajn and Panitchpakdi in \cite{AroPan56}, where it is actually shown that this property characterizes hyperconvexity. Note that any complete $\mathbb{R}$-tree is a hyperconvex metric space (see \cite{Kir88}).

\begin{theorem}[Aronszajn, Panitchpakdi \cite{AroPan56}] \label{thm-hyp}
Let $X$ be a metric space and $Y$ a hyperconvex metric space. Suppose $A \subseteq X$ is nonempty and $f:A\to Y$ is a Lipschitz mapping. Then there exists a Lipschitz extension $f':X \to Y$ of $f$ with $\emph{Lip}(f',X) = \emph{Lip}(f,A)$.
\end{theorem} 

\section{Lower semicontinuity of the multivalued extension mappings and continuous selections} \label{sect-lsc-cont-sel}

We begin this section by showing that, when considering appropriate curvature bounds on $X$ and $Y$, both mappings $\Phi$ and $\Psi$ are lower semi-continuous which is an immediate consequence of Lemmas \ref{lemma-lsc-Phi} and \ref{lemma-lsc-Psi}, respectively.  The proof strategy follows the one used for Hilbert spaces in \cite{Kop12}. 

\begin{lemma}\label{lemma-lsc-Phi}
Let $\kappa \le 0$, $X$ a CBB$(\kappa)$ space, $Y$ a complete CAT$(\kappa)$ space and $A \subseteq X$ nonempty. Let $f \in \mathcal{N}(X,Y)$. Then for every $\varepsilon > 0$ there exists $\delta > 0$ such that every $g \in \mathcal{N}(A,Y)$ with $\sup_{a \in A}d_Y(f(a),g(a)) < \delta$ admits an extension $g' \in  \mathcal{N}(X,Y)$ such that $d_\infty(f,g') \le \varepsilon$.
\end{lemma}
\begin{proof}
Since $f$ is a bounded mapping there exists $z \in Y$ and $M \ge 1$ such that $\sup_{x \in X}d_Y(z,f(x)) \le M$. Let $\varepsilon \in (0,1)$ and take $\delta = \varepsilon^2/\left(8M\right)$. Suppose $g \in \mathcal{N}(A,Y)$ with $\sup_{a \in A}d_Y(f(a),g(a)) < \delta$. 

Let $\kappa = 0$. Define the mapping $h : X \times \{(0,0)\} \cup A \times \{(0,\varepsilon)\} \to Y$ by: for  $x \in X$, $h\left(x,(0,0)\right) = f(x)$ and for $a \in A$, $h\left(a,(0,\varepsilon)\right) = g(a)$. 

Recalling (\ref{metric-direct-product}), for $x \in X$ and $a \in A$,
\begin{align*}
& d_Y\left(h\left(x,(0,0)\right),h\left(a,(0,\varepsilon)\right)\right)^2 = d_Y\left(f(x),g(a)\right)^2 \\
& \quad \le \left(d_Y\left(f(x),f(a)\right) + d_Y\left(f(a),g(a)\right)\right)^2 \\
& \quad \le d_X(x,a)^2 + \delta^2 + 4\delta M < d_X(x,a)^2 + \varepsilon^2 = d\left((x,(0,0)),(a,(0,\varepsilon))\right)^2. 
\end{align*}
This shows that $h$ is nonexpansive since both $f$ and $g$ are nonexpansive on $X$ and $A$, respectively.  Since $X \times \mathbb{R}^2$ is a CBB$(0)$ space, using Theorem \ref{thm-gen-Kirszbraun} we can extend $h$ to a nonexpansive mapping $h' : X \times \mathbb{R}^2 \to Y$. Define $g' : X \to Y$ by $g'(x) = h'\left(x,(0,\varepsilon)\right)$. Clearly, $g'$ is nonexpansive and coincides with $g$ on $A$. Moreover, for each $x \in X$,
\[d_Y(f(x),g'(x)) = d_Y\left(h'\left(x,(0,0)\right),h'\left(x,(0,\varepsilon)\right)\right) \le d\left((x,(0,0)),(x,(0,\varepsilon))\right) = \varepsilon.\]
This also shows that $g'$ is bounded.

When $\kappa < 0$, we apply the same argument to the nonexpansive mapping 
\[h : X \times \{(0,0,1)\} \cup A \times \left\{\left(0,\sinh\left(\sqrt{-\kappa}\varepsilon\right), \cosh\left(\sqrt{-\kappa}\varepsilon\right)\right)\right\} \to Y\] 
defined as: for  $x \in X$, $h\left(x,(0,0,1)\right) = f(x)$ and for $a \in A$, $h\left(a,\left(0,\sinh\left(\sqrt{-\kappa}\varepsilon\right), \cosh\left(\sqrt{-\kappa}\varepsilon\right)\right)\right) = g(a)$ which can be extended to a nonexpansive mapping $h' : X \times M_\kappa^2 \to Y$ (recall that $X \times M_\kappa^2$ is a CBB$(\kappa)$ space).
\end{proof}

\begin{lemma}\label{lemma-lsc-Psi}
Let $X$ be a CBB$(0)$ space, $Y$ a complete CAT$(0)$ space and $A \subseteq X$ nonempty. Let $f \in \mathcal{L}(X,Y)$ with $\emph{Lip}(f,A) =  \emph{Lip}(f,X)$. Then for every $\varepsilon > 0$ there exists $\delta > 0$ such that every $g \in \mathcal{L}(A,Y)$ for which $\sup_{a \in A}d_Y(f(a),g(a)) < \delta$ admits an extension $g' \in  \mathcal{L}(X,Y)$  with $\emph{Lip}(g,A) = \emph{Lip}(g',X)$ and $d_\infty(f,g') \le \varepsilon$.
\end{lemma}
\begin{proof}
Let $\varepsilon \in (0,1)$. Suppose first that $f$ is constant and equal to some $y \in Y$. Let $\delta = \varepsilon$. Then having any extension $g_1$ of $g$ to $X$ with $\text{Lip}(g,A) =  \text{Lip}(g_1,X)$, we can take $g':X \to Y$, $g'(x) = P_{B(y,\varepsilon)} \circ g_1$. 

Assume now $f$ is not constant. Let $z \in Y$ and $M > 0$ such that 
\[\sup_{x \in X}d_Y(z,f(x)) \le M.\] 
Let $s \in (0,1)$ for which 
\[\frac{1-s}{s^2} < \frac{\varepsilon^2}{32M(4M + 1)}.\]
Since $\text{Lip}(f,A) =  \text{Lip}(f,X) > 0$, there exist $x_0,y_0 \in A$ such that $d_Y(f(x_0),f(y_0)) > s\text{Lip}(f,X)d_X(x_0,y_0)$. Take
\[\delta = \min\left\{\frac{d_Y(f(x_0),f(y_0)) - s\text{Lip}(f,X)d_X(x_0,y_0)}{2},\frac{\varepsilon^2s^2}{32(4M+1)}\right\}.\]
Let $g \in \mathcal{L}(A,Y)$ with $\sup_{a \in A}d_Y(f(a),g(a)) < \delta$.
 
Suppose first $\text{Lip}(g,A) \le 2\text{Lip}(f,X)$. Then,
\begin{align*}
d_Y(g(x_0),g(y_0)) & \ge d_Y(f(x_0),f(y_0)) - d_Y(f(x_0),g(x_0)) - d_Y(f(y_0),g(y_0)) \\
& > d_Y(f(x_0),f(y_0)) - 2\delta \ge s\text{Lip}(f,X)d_X(x_0,y_0),
\end{align*}
from where $\text{Lip}(g,A) \ge s\text{Lip}(f,X)$. Let $\eta = \varepsilon/(4\text{Lip}(f,X))$ and $h : X \times \{(0,0)\} \cup A \times \{(0,\eta)\} \to Y$ be defined by: for $x \in X$, $h\left(x,(0,0)\right) = (1-s)z+sf(x)$ and for $a \in A$, $h\left(a,(0,\eta)\right) = g(a)$. Thus, for $x \in X$ and $a \in A$ we have that
\begin{align*}
& d_Y\left(h\left(x,(0,0)\right),h\left(a,(0,\eta)\right)\right)^2 = d_Y\left((1-s)z+sf(x), g(a)\right)^2\\
& \quad \le \left(d_Y\left((1-s)z+sf(x), f(a)\right) + d_Y(f(a),g(a))\right)^2\\
& \quad \le \left((1-s)M + s d_Y(f(x),f(a)) + \delta\right)^2 \\
& \quad \le \left(\delta + (1-s)M\right)^2 + s^2 \text{Lip}(f,X)^2d_X(x,a)^2 + 4sM\left(\delta+(1-s)M\right)\\
& \quad < s^2 \text{Lip}(f,X)^2 \left(d_X(x,a)^2 + \frac{(\delta + (1-s)M)(4M+1)}{s^2 \text{Lip}(f,X)^2}\right)\\
& \quad \qquad \text{since } \left(\delta + (1-s)M\right)^2 < \delta + (1-s)M \text{ and } s < 1\\
& \quad < s^2 \text{Lip}(f,X)^2 \left(d_X(x,a)^2  + \eta^2 \right) \\
& \quad \qquad \text{since } \delta + (1-s)M < \varepsilon^2s^2/(16(4M+1)) \\
& \quad \le \text{Lip}(g,A)^2 d\left((x,(0,0)),(a,(0,\eta))\right)^2.
\end{align*}
To complete the argument that $h$ is Lipschitz with smallest Lipschitz constant $\text{Lip}(g,A)$ one uses Busemann convexity in $Y$ along with the fact that the mappings $f$ and $g$ are Lipschitz and $\text{Lip}(g,A) \ge s\text{Lip}(f,X)$. Since $X \times \mathbb{R}^2$ is a CBB$(0)$ space, by Theorem \ref{thm-gen-Kirszbraun} we can extend $h$ to a Lipschitz mapping $h' : X \times \mathbb{R}^2 \to Y$ with $\text{Lip}(h', X \times \mathbb{R}^2) = \text{Lip}(g,A)$. Define $g' : X \to Y$ by $g'(x) = h'\left(x,(0,\eta)\right)$. Clearly, $g'$ extends $g$ and $\text{Lip}(g',X) = \text{Lip}(g,A)$. Moreover, for every $x \in X$,
\begin{align*}
d_Y(g'(x),f(x)) & \le d_Y\left(g'(x),(1-s)z+sf(x)\right) + d_Y\left((1-s)z+sf(x),f(x)\right)\\
& \le d_Y\left(h'\left(x,(0,\eta)\right),h'\left(x,(0,0)\right)\right) + (1-s)M < \text{Lip}(g,A)\eta + \varepsilon/2 \\
& \le 2\text{Lip}(f,X)\frac{\varepsilon}{4\text{Lip}(f,X)} + \frac{\varepsilon}{2} = \varepsilon.
\end{align*}

If $\text{Lip}(g,A) > 2\text{Lip}(f,X)$, consider the set
\[\tilde{A} = \left\{x \in X : \text{dist}(x,A) \ge \frac{2\delta}{\text{Lip}(g,A)}\right\}\]
and define the mapping $\tilde{g} : A \cup \tilde{A} \to Y$ by: for $a \in A$, $\tilde{g}(a) = g(a)$ and for $x \in \tilde{A}$, $\tilde{g}(x) = f(x)$. To see that $\text{Lip}(g,A) = \text{Lip}(\tilde{g}, A \cup \tilde{A})$ it suffices to verify that for any $a \in A$ and $x \in \tilde{A}$,
\begin{align*}
d_Y\left(\tilde{g}(x),\tilde{g}(a)\right) & = d_Y(f(x),g(a)) \le d_Y(f(x),f(a)) + d_Y(f(a),g(a))\\
& < \frac{\text{Lip}(g,A)}{2}d_X(x,a) + \delta\le \frac{\text{Lip}(g,A)}{2}d_X(x,a) + \frac{\text{Lip}(g,A)}{2}\text{dist}(x,A) \\
& \le \text{Lip}(g,A)d_X(x,a).
\end{align*}
Take $g'$ to be any extension of $\tilde{g}$ for which $\text{Lip}(g,A) = \text{Lip}(g',X)$. For $x \in \tilde{A}$, $f(x) = g'(x)$. If $x \notin \tilde{A}$, there exists $a \in A$ such that $d_X(x,a) < 2\delta/\text{Lip}(g,A)$. Thus,
\begin{align*}
d_Y(f(x),g'(x)) & \le d_Y(f(x),f(a)) + d_Y(f(a),g'(a)) + d_Y(g'(a),g'(x))\\
& < \frac{\text{Lip}(g,A)}{2}\frac{2\delta}{\text{Lip}(g,A)} + \delta + \text{Lip}(g,A)\frac{2\delta}{\text{Lip}(g,A)} = 4 \delta < \varepsilon. 
\end{align*}
This ends the proof.
\end{proof}

\begin{theorem} \label{thm-lsc-Phi}
Let $\kappa \le 0$, $X$ a CBB$(\kappa)$ space, $Y$ a complete CAT$(\kappa)$ space and $A \subseteq X$ nonempty. Then the mapping $\Phi : \mathcal{N}(A,Y) \to \oldcal{P}\left(\mathcal{N}(X,Y)\right)$ is lower semi-continuous.
\end{theorem}

\begin{theorem}\label{thm-lsc-Psi}
Let $X$ be a CBB$(0)$ space, $Y$ a complete CAT$(0)$ space and $A \subseteq X$ nonempty. Then the mapping $\Psi : \mathcal{L}(A,Y) \to \oldcal{P}\left(\mathcal{L}(X,Y)\right)$ is lower semi-continuous.
\end{theorem}

Using the lower semi-continuity of the mappings $\Phi$ and $\Psi$ we prove that they admit continuous selections. In order to obtain these singlevalued continuous extension operators we apply a selection result due to Horvath \cite{Hor91} which is a generalization of the classical Michael selection theorem to the setting of $c$-spaces. Before stating this selection result we recall the following notions: for $Z$ a topological space, denote by $\langle Z \rangle$ the family of its nonempty and finite subsets. A mapping $F : \langle Z \rangle \to \oldcal{P}(Z)$ is a {\it $c$-structure} if firstly, for each $A \in \langle Z \rangle$, $F(A)$ is nonempty and contractible, and secondly, for every $A_1,A_2 \in \langle Z \rangle$, $A_1 \subseteq A_2$ implies $F(A_1) \subseteq F(A_2)$. The pair $(Z,F)$ is called a {\it $c$-space} and $V \subseteq Z$ is an {\it $F$-set} if for every $A \in \langle V \rangle$ we have that $F(A) \subseteq V$. A $c$-space $(Z,F)$ is called an {\it l.c. metric space} is $(Z,d)$ is a metric space such that open balls are $F$-sets and if $V \subseteq Z$ is an $F$-set, then for every $\varepsilon > 0$, $\{z \in Z : \text{dist}(z,V) < \varepsilon\}$ is an $F$-set. The selection result that we apply is the following.

\begin{theorem}[Horvath \cite{Hor91}] \label{thm-Horvath-sel}
Let $U$ be a paracompact topological space, $(Z,F)$ an l.c. complete metric space and $\Gamma:U \to \oldcal{P}(Z)$ lower semi-continuous such that for each $u \in U$, $\Gamma(u)$ is a nonempty and closed $F$-set. Then there exists a continuous selection for $\Gamma$.
\end{theorem}

Let $\kappa \le 0$. Suppose $X$ is a CBB$(\kappa)$ space and $Y$ a complete CAT$(\kappa)$ space. We check in the sequel that we can indeed make use of the above theorem relying basically on Busemann convexity in $Y$. We say that $B \in \oldcal{P}(C(X,Y))$ is {\it convex} if for every $g_1,g_2 \in B$ and every $t \in [0,1]$ we have that the mapping $h : X \to Y$, $h = (1-t)g_1 + tg_2$ (that is, $h(x) = (1-t)g_1(x) + tg_2(x)$ for every $x \in X$) belongs to $B$. Note that balls in $C(X,Y)$ are convex.

The mapping $\Phi$ has nonempty and closed values in $C(X,Y)$. Moreover, for each $f \in \mathcal{N}(A,Y)$, $\Phi(f)$ is convex. To see this let $f',f'' \in \Phi(f)$ and $t \in [0,1]$. Then, for each $x \in X$,
\begin{align*}
d_Y((1-t)f'(x) + tf''(x),(1-t)f'(y) + tf''(y)) &\le (1-t)d_Y(f'(x),f'(y))\\
& \quad + td_Y(f''(x),f''(y))\\
& \le d_X(x,y). 
\end{align*}
Similarly, when $\kappa = 0$, $\Psi$ is also nonempty, closed and convex-valued. Define $F : \langle C(X,Y) \rangle \to \oldcal{P}(C(X,Y))$ by
\[
F(A) = \bigcap \{B : A \subseteq B, B \text{ convex}\}, \quad \text{for each } A \in \langle C(X,Y) \rangle.
\]
Let $A \in \langle C(X,Y) \rangle$. Then $F(A) \ne \emptyset$. Fix $g_1 \in A$ and define  $H : [0,1] \times F(A) \to F(A)$ by $H(t,f) = (1-t)f + tg_1$. Note that for each $f \in F(A)$, $H(0,f) = f$ and $H(1,f) = g_1$. It is easy to see that $H$ is continuous and so $F(A)$ is contractible. Clearly, for every $A_1, A_2 \in \langle C(X,Y) \rangle$, $A_1 \subseteq A_2$ implies $F(A_1) \subseteq F(A_2)$. Thus, $\left(C(X,Y),F\right)$ is a $c$-space. Note that a subset of $C(X,Y)$ is an $F$-set if and only if it is convex. By Busemann convexity in $Y$ one can finally show that $\left(C(X,Y),F\right)$ is an l.c. metric space.

\begin{theorem}\label{thm-cont-sel-nonexp}
Let $\kappa \le 0$, $X$ a CBB$(\kappa)$ space, $Y$ a complete CAT$(\kappa)$ space and $A \subseteq X$ nonempty. Then there exists a continuous mapping $\alpha : \mathcal{N}(A,Y) \to \mathcal{N}(X,Y)$ such that for all $g \in \mathcal{N}(A,Y)$, $\alpha(g)(a) = g(a)$ for every $a \in A.$
\end{theorem}
\begin{proof}
We can view the mapping $\Phi$ with values in $\oldcal{P}(C(X,Y))$ while still preserving its lower semi-continuity. Since any metric space is a paracompact topological space we can now apply Theorem \ref{thm-Horvath-sel} to obtain a continuous extension mapping $\alpha : \mathcal{N}(A,Y) \to C(X,Y)$. Because $\Phi$ actually takes values in $\oldcal{P}(\mathcal{N}(X,Y))$ we obtain the conclusion.
\end{proof}

For bounded Lipschitz mappings we obtain the following result.
\begin{theorem}\label{thm-cont-sel-Lip}
Let $X$ be a CBB$(0)$ space, $Y$ a complete CAT$(0)$ space and $A \subseteq X$ nonempty. Then there exists a continuous mapping $\beta : \mathcal{L}(A,Y) \to \mathcal{L}(X,Y)$ such that for all $g \in \mathcal{L}(A,Y)$, $\beta(g)(a) = g(a)$ for every $a \in A$ and $\emph{Lip}(\beta(g),X) = \emph{Lip}(g,A)$.
\end{theorem}

\begin{remark} \label{rmk-R-trees}
Note that in Lemmas \ref{lemma-lsc-Phi} and \ref{lemma-lsc-Psi} the lower curvature bound of $X$ is only used to apply Theorem \ref{thm-gen-Kirszbraun}. However, when $Y$ is a complete $\mathbb{R}$-tree, one can extend Lipschitz mappings (while keeping the same Lipschitz constant) if $X$ is an arbitrary metric space. Thus, as before, one can consider even a simpler reasoning in $X \times \mathbb{R}$ to obtain that both mappings $\Phi$ and $\Psi$ are lower semi-continuous and admit continuous selections. This property will be improved for the mapping $\Phi$ in Section \ref{hyperconvexity}.
\end{remark}

\begin{remark}
If $\kappa > 0$ the argument given in this section does not work in a straightforward way. Note that in this case the direct product $X \times M_\kappa^2$ is not necessarily a CBB$(\kappa)$ space. Moreover, the images of the mappings $\Phi$ and $\Psi$ are no longer $F$-sets when considering the $c$-structure $F$ defined before.
\end{remark}

\section{A convexity assumption on the images of the extensions}\label{convexity}
 
In this section we show that one can actually choose extensions in a continuous way even when imposing the condition that the image of the extension belongs to the closure of the convex hull of the image of the original mapping. Related results in the case of Hilbert spaces were recently established in \cite{KopRei11}, and we extend them to our setting.

Recall first the next inequality which stems from the work of Reshetnyak (see, for instance, \cite[Theorem 2.3.1]{Jos97} or \cite[ Lemma 2.1]{LanPavSch00} for a simple proof).

\begin{lemma} \label{lemma-ineq-CAT(0)}
Let $Y$ be a CAT$(0)$ space. Then for every $x,y,u,v \in Y,$
\[d(x,y)^2 + d(u,v)^2 \le d(x,v)^2 + d(y,u)^2 + 2d(x,u)d(y,v).\]
\end{lemma}

The property below provides a uniform bound on the distance between the projection points from a common point onto two sets. A similar result in uniformly smooth Banach spaces is \cite[Lemma 3.4]{AlbReiYao03}.

\begin{lemma}\label{lemma-projections-balls}
Let $Y$ be a complete CAT$(0)$ space, $C_1, C_2 \subseteq Y$ nonempty, closed and convex and suppose $r_1$ and $r_2$ are positive numbers. If there exists $z \in Y$ such that $C_1, C_2 \subseteq B(z,r_1)$, then for any $x \in B(z,r_2)$,
\[d(P_{C_1}(x), P_{C_2}(x))^2 \le 2(r_1 + r_2)H(C_1,C_2).\]
\end{lemma}
\begin{proof} 
Let $C_1,C_2 \subseteq B(z,r_1)$ and $x \in B(z,r_2)$. Denote $p_1 = P_{C_1}(x)$, $p_2=P_{C_2}(x)$, $q_1=P_{C_1}(p_2)$ and $q_2=P_{C_2}(p_1)$. Clearly, $d(p_2,q_1) \le H(C_1,C_2)$ and $d(p_1,q_2) \le H(C_1,C_2)$. Note also that that $d(x,p_1) \le d(x,z) + d(z,p_1) \le r_1+r_2$ and $d(x,p_2) \le r_1 + r_2$. Since $q_1 \in C_1$ and $p_1=P_{C_1}(x)$ it follows that $\angle_{p_1}(x,q_1) \ge \pi/2$ (see \cite[Proposition 2.4, page 176]{Bri99}) which yields 
$d(x,q_1)^2 \ge d(x,p_1)^2 + d(p_1,q_1)^2$. Similarily, $d(x,q_2)^2 \ge d(x,p_2)^2 + d(p_2,q_2)^2$. By Lemma \ref{lemma-ineq-CAT(0)} we also have that
\[d(x,q_1)^2 + d(p_1,p_2)^2 \le d(x,p_2)^2 + d(p_1,q_1)^2 + 2d(x,p_1)d(p_2,q_1),\]
and therefore
\begin{equation}\label{lemma-proj-eq1}
d(x,p_1)^2 + d(p_1,p_2)^2 \le d(x,p_2)^2 + 2d(x,p_1)d(p_2,q_1).
\end{equation}
Likewise,
\begin{equation}\label{lemma-proj-eq2}
d(x,p_2)^2 + d(p_1,p_2)^2 \le d(x,p_1)^2 + 2d(x,p_2)d(p_1,q_2).
\end{equation}
Adding (\ref{lemma-proj-eq1}) and (\ref{lemma-proj-eq2}) we get that
\[d(p_1,p_2)^2 \le d(x,p_1)d(p_2,q_1) + d(x,p_2)d(p_1,q_2) \le 2\left(r_1+r_2\right)H(C_1,C_2).\]
\end{proof}

We prove next a property of the Hausdorff distance. For the corresponding result in the setting of normed spaces, see \cite{Pri40}.
\begin{lemma}\label{lemma-Hausdorff-dist}
Let $Y$ be a CAT$(0)$ space, $C_1, C_2 \subseteq Y$ nonempty. Then,
\[H\left(\overline{\emph{co}}(C_1), \overline{\emph{co}}(C_2)\right) \le H(C_1,C_2).\]
\end{lemma}
\begin{proof}
Let $c \in C_2$. Obviously, $\text{dist}\left(c,\overline{\text{co}}(C_1)\right) \le \text{dist}\left(c,C_1\right) \le H(C_1,C_2)$. Consider the set
\[E = \left\{y \in Y : \text{dist}\left(y,\overline{\text{co}}(C_1)\right) \le H(C_1,C_2)\right\},\]
which is a closed and convex set (by Busemann convexity). Since $C_2 \subseteq E$ it follows that $\overline{\text{co}}(C_2) \subseteq E$, from where $\sup_{c \in \overline{\text{co}}(C_2)} \text{dist}\left(c,\overline{\text{co}}(C_1)\right) \le H(C_1,C_2)$.
In a similar way we have that $\sup_{c \in \overline{\text{co}}(C_1)} \text{dist}\left(c,\overline{\text{co}}(C_2)\right) \le H(C_1,C_2)$ and we are done.
\end{proof}

\begin{theorem} \label{thm-cont-sel-nonexp-cv}
Let $\kappa \le 0$, $X$ a CBB$(\kappa)$ space, $Y$ a complete CAT$(\kappa)$ space and $A \subseteq X$ nonempty. Then there exists a continuous mapping $\alpha_c : \mathcal{N}(A,Y) \to \mathcal{N}(X,Y)$ such that for all $g \in \mathcal{N}(A,Y)$, $\alpha_c(g)(a) = g(a)$ for every $a \in A$ and $\alpha_c(g)(X) \subseteq \overline{\emph{co}}\left(g(A)\right)$.
\end{theorem}
\begin{proof}
By Theorem \ref{thm-cont-sel-nonexp} there exists a continuous $\alpha : \mathcal{N}(A,Y) \to \mathcal{N}(X,Y)$ such that for all $g \in \mathcal{N}(A,Y)$, $\alpha(g)$ extends $g$. Define a mapping $\alpha_c$ on $\mathcal{N}(A,Y)$ by
\[\alpha_c(g)(x) = P_{\overline{\text{co}}\left(g(A)\right)}\left(\alpha(g)(x)\right), \quad \text{ for each } g \in \mathcal{N}(A,Y) \text{ and } x \in X.\]
For each $g \in \mathcal{N}(A,Y)$, $\alpha_c(g) \in \mathcal{N}(X,Y)$ since the projection onto complete and convex subsets is nonexpansive. Clearly, $\alpha_c(g)(X) \subseteq \overline{\text{co}}\left(g(A)\right)$ and $\alpha_c(g)$ coincides with $g$ on $A$. 

Thus, we only need to prove that $\alpha_c$ is continuous. Let $f \in \mathcal{N}(A,Y)$ and $\varepsilon  > 0$. Since $\alpha$ is continuous, there exists $\delta_1 < 1$ such that for every $g \in \mathcal{N}(A,Y)$ with $d_\infty(f,g) < \delta_1$ we have that $d_\infty\left(\alpha(f),\alpha(g)\right) < \varepsilon/2$. Fix $z \in Y$. Let $r = \sup_{x \in X}d_Y\left(z,\alpha(f)(x)\right)$ and take $\delta = \min\left\{\delta_1, \frac{\varepsilon^2}{16(r+1)}\right\}.$ Let $g \in \mathcal{N}(A,Y)$ with $d_\infty(f,g) < \delta$. Then, for every $x \in X$,
\begin{align*}
d_Y\left(\alpha_c(f)(x),\alpha_c(g)(x)\right) & = d_Y\left(P_{\overline{\text{co}}\left(f(A)\right)}\left(\alpha(f)(x)\right),P_{\overline{\text{co}}\left(g(A)\right)}\left(\alpha(g)(x)\right)\right) \\
& \le d_Y\left(P_{\overline{\text{co}}\left(f(A)\right)}\left(\alpha(f)(x)\right),P_{\overline{\text{co}}\left(g(A)\right)}\left(\alpha(f)(x)\right)\right)\\ 
& \quad + d_Y\left(P_{\overline{\text{co}}\left(g(A)\right)}\left(\alpha(f)(x)\right),P_{\overline{\text{co}}\left(g(A)\right)}\left(\alpha(g)(x)\right)\right).
\end{align*}
Note that $\sup_{a \in A}d_Y(z,f(a)) \le r$ and $\sup_{a \in A}d_Y(z,g(a)) \le r + 1$. Apply Lemma \ref{lemma-projections-balls} with $C_1 = \overline{\text{co}}\left(f(A)\right)$, $C_2 = \overline{\text{co}}\left(g(A)\right)$ and $r_1=r_2=r+1$ to get that
\begin{align*}
d_Y\left(P_{\overline{\text{co}}\left(f(A)\right)}\left(\alpha(f)(x)\right),P_{\overline{\text{co}}\left(g(A)\right)}\left(\alpha(f)(x)\right)\right) & \le 2\sqrt{r+1}\sqrt{H\left(\overline{\text{co}}\left(f(A)\right), \overline{\text{co}}\left(g(A)\right)\right)}\\
& \le 2\sqrt{r+1}\sqrt{H\left(f(A),g(A)\right)}\\
& \qquad \text{by Lemma } \ref{lemma-Hausdorff-dist}\\
& \le 2\sqrt{r+1} \sqrt{\sup_{a \in A}d_Y(f(a),g(a))} \le \varepsilon/2.
\end{align*}
At the same time,
\begin{align*}
d_Y\left(P_{\overline{\text{co}}\left(g(A)\right)}\left(\alpha(f)(x)\right),P_{\overline{\text{co}}\left(g(A)\right)}\left(\alpha(g)(x)\right)\right) & \le d_Y\left(\alpha(f)(x),\alpha(g)(x)\right)\\
& \le d_\infty\left(\alpha(f),\alpha(g)\right) < \varepsilon/2.
\end{align*}
Hence, $d_\infty\left(\alpha_c(f),\alpha_c(g)\right) < \varepsilon$ which proves that $\alpha_c$ is continuous too.
\end{proof}

Following the same idea of proof one can give an analogous result for bounded Lipschitz mappings.

\begin{theorem} \label{thm-cont-sel-Lip-cv}
Let $X$ be a CBB$(0)$ space, $Y$ a complete CAT$(0)$ space and $A \subseteq X$ nonempty. Then there exists a continuous mapping $\beta_c : \mathcal{L}(A,Y) \to \mathcal{L}(X,Y)$ such that for all $g \in \mathcal{L}(A,Y)$, $\beta_c(g)(a) = g(a)$ for every $a \in A$, $\emph{Lip}(\beta_c(g),X) = \emph{Lip}(g,A)$ and $\beta_c(g)(X) \subseteq \overline{\emph{co}}\left(g(A)\right)$.
\end{theorem}

In fact one can also consider the multivalued extension mappings:
\begin{itemize}
\item $\Phi_c : \mathcal{N}(A,Y) \to \oldcal{P}\left(\mathcal{N}(X,Y)\right)$ which assigns to each nonexpansive mapping $f \in \mathcal{N}(A,Y)$ all its nonexpansive extensions $f' \in \mathcal{N}(X,Y)$ with $f'(X) \subseteq \overline{\text{co}}(f(A))$.
\item $\Psi_c : \mathcal{L}(A,Y) \to \oldcal{P}\left(\mathcal{L}(X,Y)\right)$ which assigns to each Lipschitz mapping $f \in \mathcal{L}(A,Y)$ all its Lipschitz extensions $f' \in \mathcal{L}(X,Y)$ with $\text{Lip}(f,A) = \text{Lip}(f',X)$ and $f'(X) \subseteq \overline{\text{co}}(f(A))$.
\end{itemize}

These mappings, too, will be lower semi-continuous.

\begin{theorem} \label{thm-lsc-Phi-c}
Let $\kappa \le 0$, $X$ a CBB$(\kappa)$ space, $Y$ a complete CAT$(\kappa)$ space and $A \subseteq X$ nonempty. Then the mapping $\Phi_c : \mathcal{N}(A,Y) \to \oldcal{P}\left(\mathcal{N}(X,Y)\right)$ is lower semi-continuous.
\end{theorem}
\begin{proof}
We show that for every $f \in \mathcal{N}(X,Y)$ with $f(X) \subseteq \overline{\text{co}}(f(A))$ and for every $\varepsilon > 0$ there exists $\delta > 0$ such that every $g \in \mathcal{N}(A,Y)$ with $\sup_{a \in A}d_Y(f(a),g(a)) < \delta$ admits an extension $g' \in  \mathcal{N}(X,Y)$ with $g'(X) \subseteq \overline{\text{co}}(g(A))$ and $d_\infty(f,g') \le \varepsilon$.

Let $f$ be as above and $\varepsilon > 0$. By Lemma \ref{lemma-lsc-Phi} there exists $\delta > 0$ such that every $g \in \mathcal{N}(A,Y)$ with $\sup_{a \in A}d_Y(f(a),g(a)) < \delta$ admits an extension $g_1 \in  \mathcal{N}(X,Y)$ with $d_\infty(f,g_1) \le \varepsilon/3$.

Define $g':X \to Y$, $g'(x) = P_{\overline{\text{co}}\left(g(A)\right)}\left(g_1(x)\right)$. Clearly, $g'$ is nonexpansive, extends $g$ and $g'(X) \subseteq \overline{\text{co}}(g(A))$. Let $x \in X$. Then,
\begin{equation} \label{thm-lsc-Psi-c-eq1}
d_Y(f(x),g'(x)) \le d_Y(f(x),g_1(x)) + d_Y(g_1(x),g'(x)) \le \varepsilon/3 + d_Y(g_1(x),g'(x)).
\end{equation}
For every $y \in \overline{\text{co}}(g(A))$ we have that
\[d_Y(g_1(x),g'(x)) \le d_Y(g_1(x),y) \le d_Y(g_1(x),f(x)) + d_Y(f(x),y),\]
from where
\[d_Y(g_1(x),g'(x)) \le \varepsilon/3 + \text{dist}\left(f(x),\overline{\text{co}}(g(A))\right).\]
Consider $E = \left\{y \in Y : \text{dist}\left(y,\overline{\text{co}}(g(A))\right) \le \varepsilon/3\right\}$. We know that $f(A) \subseteq E$ since for any $a \in A$, 
\[\text{dist}\left(f(a),\overline{\text{co}}(g(A))\right) \le d_Y(f(a),g(a)) \le \varepsilon/3.\]
Since $E$ is closed and convex we have that $\overline{\text{co}}(f(A)) \subseteq E$. But $f(X) \subseteq \overline{\text{co}}(f(A))$ and so $f(x) \in E$. Thus, $\text{dist}\left(f(x),\overline{\text{co}}(g(A))\right) \le \varepsilon/3$ which implies that $d_Y(g_1(x),g'(x)) \le 2\varepsilon/3$. Using (\ref{thm-lsc-Psi-c-eq1}), we obtain that $d_Y(f(x),g'(x)) \le \varepsilon$.
\end{proof}

The same argument yields the result for bounded Lipschitz mappings.

\begin{theorem}\label{thm-lsc-Psi-c}
Let $X$ be a CBB$(0)$ space, $Y$ a complete CAT$(0)$ space and $A \subseteq X$ nonempty. Then the mapping $\Psi_c : \mathcal{L}(A,Y) \to \oldcal{P}\left(\mathcal{L}(X,Y)\right)$ is lower semi-continuous.
\end{theorem}

Note that one could apply, as in Section \ref{sect-lsc-cont-sel}, Theorem \ref{thm-Horvath-sel} to the mappings $\Phi_c$ and $\Psi_c$ to obtain directly Theorems \ref{thm-cont-sel-nonexp-cv} and \ref{thm-cont-sel-Lip-cv}, respectively.

\begin{remark} 
Similar results to the ones given in this section can be proved when $Y$ is a complete $\mathbb{R}$-tree and $X$ is a general metric space.
\end{remark}

\section{Nonexpansive selections in hyperconvex metric spaces}\label{hyperconvexity}

We prove next that results for mappings $\Phi$ and $\Phi_c$ from previous sections  can be strengthened if the target space is a hyperconvex metric space.  Indeed, we will show that extensions of nonexpansive (Lipschitz) mappings can be chosen to be not only continuously but in a nonexpansive (or Lipschitz) way. Our result will follow as a direct application of the next result \cite[Theorem 1]{KhaKirMar00} (see also \cite[Theorem 1]{Sin89}). The class of externally hyperconvex subsets of a metric space $X$, defined in Section \ref{prelim}, is denoted by ${\mathcal E}(X)$.

\begin{theorem}[Khamsi, Kirk, Mart\'inez-Y\'a\~{n}ez \cite{KhaKirMar00}]\label{Kha00}
Let $X$ be a metric space and $Y$ a hyperconvex metric space. If $T\colon X\to {\mathcal E}(Y)$ is a multivalued mapping, then there exists a selection mapping $f\colon X\to Y$ of $T$ such that $d(f(x),f(y))\le H(T(x),T(y))$.
\end{theorem}

This theorem implies in particular that if the multivalued mapping $T$ is nonexpansive (Lipschitz) then $f$ can be chosen nonexpansive (Lipschitz). The next fact we need is that the set ${\mathcal N} (X,Y)$ endowed with the supremum distance is hyperconvex. This is basically due to \cite[Theorem 3]{KhaKirMar00} where the result is proved for ${\mathcal N} (Y,Y)$ with $Y$ hyperconvex, but the proof carries over with no modification to our case. We state the result as it is given in \cite{KhaKirMar00}.

\begin{theorem}[Khamsi, Kirk, Mart\'inez-Y\'a\~{n}ez \cite{KhaKirMar00}]\label{3Kha00}
Let $Y$ be hyperconvex and $\lambda >0$. Let $\lambda (Y,Y)$ be the family of all bounded $\lambda$-Lipschitz mappings from $Y$ into $Y$. Then $\lambda (Y,Y)$ is hyperconvex endowed with the supremum distance.
\end{theorem}

Next we show that the mapping $\Phi$ is nonexpansive.

\begin{lemma}\label{set-nonexp}
Let $X$ be a metric space, $A\subseteq X$ nonempty and $Y$ a hyperconvex metric space. Then $\Phi$ is a nonexpansive multivalued mapping.
\end{lemma} 

\begin{proof}
Let $f,g\in {\mathcal N}(A, Y)$ and $f'\in \Phi (f)$. We need to show that there exists $g'\in \Phi (g)$ such that $d_\infty (f',g')\le d_\infty(f,g)$. We will construct $g'$ point by point beginning with $g'(a)=g(a)$ for all $a\in A$. Let $r=d_\infty(f,g)$, $x_0\in X\setminus A$ and consider the intersection
\[A_0=\left(\bigcap_{a\in A}B(g'(a),d(a,x_0))\right)\bigcap B(f'(x_0),r).\]
It is easy to see that these balls have nonempty intersection two-by-two and so, by hyperconvexity of $Y$, $A_0 \ne \emptyset$. Define $g'(x_0)$ as any element in $A_0$.

Let $x_1\in X\setminus (A\cup \{x_0\})$ and consider now
\[A_1=\left(\bigcap_{a\in A}B(g'(a),d(a,x_1))\right)\bigcap B(g'(x_0),d(x_0,x_1))\bigcap B(f'(x_1),r).\]
Checking distances between centers, with $a, a_1, a_2\in A$, we have that:
\begin{align*}
d(g'(a_1),g'(a_2)) &\le d(a_1,a_2)\le d(a_1,x_1)+d(a_2,x_1),\\
d(g'(a),g'(x_0))&\le d(a,x_0)\le d(a,x_1)+d(x_0,x_1),\\
d(g'(a),f'(x_1))&\le d(g'(a),f'(a))+d(f'(a),f'(x_1))\le r+d(a,x_1),\\
d(g'(x_0),f'(x_1))&\le d(g'(x_0),f'(x_0))+d(f'(x_0),f'(x_1))\le r+d(x_0,x_1).
\end{align*}
Hence, by hyperconvexity of $Y$, we have that $A_1$ is nonempty. Choose $g'(x_1)$ as any point in $A_1$. The proof is completed after a standard transfinite argument. We omit further details.
\end{proof}

 
 
To be able to apply Theorem \ref{Kha00} we still need the values of $\Phi$ to be externally hyperconvex. 
 
\begin{lemma} \label{exthyp}
Let $X$ be a metric space, $A\subseteq X$ nonempty and $Y$ a hyperconvex metric space. Then $\Phi (f)$ is externally hyperconvex in ${\mathcal N}(X,Y)$ for every $f\in {\mathcal N}(A,Y)$.
\end{lemma}
 
\begin{proof}
We know that $\Phi (f)$ is nonempty due to Theorem \ref{thm-hyp}. Let $f\in {\mathcal N}(A,Y)$, $\{ f_\alpha\}_{\alpha\in{\mathcal A}} \subseteq {\mathcal N}(X,Y)$ and $\{r_\alpha\}_{\alpha\in\mathcal A}\subseteq {\mathbb R}^+$ such that $d_\infty (f_\alpha,f_\beta)\le r_\alpha+r_\beta$ and ${\rm dist}(f_\alpha, \Phi(f))\le r_\alpha$ for all $\alpha,\beta\in\mathcal A$. We need to prove that
\[\left( \bigcap_{\alpha\in{\mathcal A}} B(f_\alpha,r_\alpha)\right) \bigcap \Phi (f)\neq \emptyset.\]
We will construct an extension $f'$ of $f$ in the above intersection by transfinite induction. Let $f'(a)=f(a)$ for $a\in A$. Since ${\rm dist}(f_\alpha, \Phi(f))\le r_\alpha$ it is clear that $d(f_\alpha (a),f'(a))\le r_\alpha$ for every $a\in A$. Let $x_0\in X\setminus A$ and consider the intersection
\[A_0=\left( \bigcap_{a\in{ A}} B(f'(a),d(a,x_0))\right) \bigcap \left( \bigcap_{\alpha\in{\mathcal A}} B(f_\alpha(x_0),r_\alpha)\right).\]
A two-by-two case study and the hyperconvexity of $Y$ (see the $A_1$ case below for more details) easily show that $A_0$ is nonempty. Define $f'(x_0)$ as any point in $A_0$. 

Taking now $x_1\in X\setminus (A\cup \{x_0\})$, the corresponding intersection to look at is 
\[A_1=\left( \bigcap_{a\in{ A}} B(f'(a),d(a,x_1))\right)\bigcap B(f'(x_0),d(x_0,x_1)) \bigcap \left( \bigcap_{\alpha\in{\mathcal A}} B(f_\alpha (x_1),r_\alpha)\right).\]
We check next the hyperconvexity condition for $A_1$ ($a, a_1,  a_2$ stand for points in $A$):
\begin{align*}
d(f'(a_1),f'(a_2)) &\le d(a_1,a_2)\le d(a_1,x_1)+d(a_2,x_1),\\
d(f'(a),f'(x_0))&\le d(a,x_0)\le d(a,x_1)+d(x_0,x_1),\\
d(f'(a),f_\alpha(x_1))&\le d(f'(a),f_\alpha(a))+d(f_\alpha(a),f_\alpha(x_1))\le r_\alpha+d(a,x_1),\\
d(f'(x_0),f_\alpha(x_1))&\le d(f'(x_0),f_\alpha(x_0))+d(f_\alpha(x_0),f_\alpha(x_1))\le r_\alpha+d(x_0,x_1),\\
d(f_\alpha (x_1),f_\beta(x_1))&\le r_\alpha+r_\beta.
\end{align*}
Therefore, $A_1$ is nonempty and it suffices to define $f'(x_1)$ as any point in $A_1$. The proof is completed by transfinite induction.
\end{proof} 
 
\begin{remark}
Regarding the mapping $\Psi$, it does not seem that the approach applied to $\Phi$ in the hyperconvex case is also working.  
\end{remark}

\begin{remark} 
Notice that Lemma \ref{exthyp} improves \cite[Theorem 17]{Sin89}.
\end{remark}

Now, we can give the main result of this section.

\begin{theorem}\label{main5}
Let $X$ be a metric space, $A\subseteq X$ nonempty and $Y$ a hyperconvex metric space. Then there exists a nonexpansive mapping $\alpha\colon {\mathcal N}(A,Y)\to {\mathcal N}(X,Y)$ such that for all $g\in {\mathcal N}(A,Y)$, $\alpha (g)(a)=g(a)$ for every $a\in A$. 
\end{theorem}
\begin{proof}
It directly follows now from Lemmas \ref{set-nonexp} and \ref{exthyp}, and Theorems \ref{Kha00} and \ref{3Kha00}. 
\end{proof}

\begin{remark}
By considering adequate modifications in the proofs, Theorem \ref{main5} also holds for $\lambda (A,Y)$ instead of ${\mathcal N}(A,Y)$ and $\lambda$-Lipschitz extensions instead of nonexpansive extensions. 
\end{remark}

In particular, since complete $\mathbb R$-trees are hyperconvex spaces, we have the following corollary which improves the corresponding result from Section 3 (see Remark \ref{rmk-R-trees}).

\begin{corollary}
Let $X$ be a metric space, $A\subseteq X$ nonempty and $Y$ a complete $\mathbb R$-tree. Then the multivalued mapping $\Phi$ admits a nonexpansive selection.
\end{corollary}

It is now natural to wonder about results from Section \ref{convexity} in the hyperconvex setting. First, we need to clarify the notion of convex hull. A natural option in this case is to consider the admissible hull of a set.

\begin{definition}
Let $X$ be a metric space and $A\subseteq X$ nonempty and bounded. Then the admissible hull ${\rm cov}(A)$ of $A$ is the intersection of all the closed balls containing $A$. A set is called admissible if it coincides with its admissible hull.
\end{definition}

It is easy to see that 
\[{\rm cov}(A)=\bigcap_{x\in X}B(x,r_x(A)),\]
where $r_x(A)=\sup\{d(x,a)\colon a\in A\}$. Now we can define $\Phi_c$ as in Section \ref{convexity} replacing $\overline{\rm co}(f(A))$ with ${\rm cov}(f(A))$. In order to prove Theorem \ref{main5} for $\Phi_c$ in the hyperconvex setting we only need to show that Lemmas \ref{set-nonexp} and \ref{exthyp} still hold true. This is indeed the case. We point out next how to modify the corresponding proofs.

\begin{lemma}
Let $X$ be a metric space, $A\subseteq X$ nonempty and $Y$ a hyperconvex metric space. Then $\Phi_c$ is a nonexpansive multivalued mapping.
\end{lemma} 

\begin{proof}
First we need to show that $\Phi_c(f)$ is nonempty for $f\in{\mathcal N}(A,Y)$ which directly follows from the fact that admissible subsets of hyperconvex spaces are hyperconvex themselves and so, from Theorem \ref{thm-hyp}, extensions $f'$ of $f$ exist such that $f'\in{\mathcal N}(X,{\rm cov}(f(A)))$. Now, define the set $A_0$ as in the proof of Lemma \ref{set-nonexp} in the following way:
\[A_0=\left(\bigcap_{a\in A}B(g'(a),d(a,x_0))\right)\bigcap \left(\bigcap_{y\in Y}B(y,r_y(g(A)))\right)\bigcap B(f'(x_0),r).\]
To apply hyperconvexity the only case which is not trivial is for pairs of balls centered at $y\in Y$ and at $f'(x_0)$. But for this case we have that, since $f'\in{\mathcal N}(X,{\rm cov}(f(A)))$, $d(y,f'(x_0))\le r_y(f(A))=\sup_{a\in A} d(y,f(a))\le \sup_{a\in A} (d(y,g(a))+d(g(a),f(a)))\le r_y(g(A))+r$. Therefore $A_0\neq \emptyset$ and we may define $g'(x_0)$ as any point in $A_0$.

For the next step we need to consider
\[A_1=A'_1\cap B(g'(x_0),d(x_0,x_1)),\]
where
\[A'_1 = \left(\bigcap_{a\in A}B(g'(a),d(a,x_1))\right)\bigcap \left(\bigcap_{y\in Y}B(y,r_y(g(A)))\right)\bigcap B(f'(x_1),r).\] 
Now, since $g'(a)\in g(A)$ the case $d(g'(a),y)$ for $a\in A$ and $y\in Y$ follows. The case $d(y,g'(x_0))$ follows by the previous step. The last case, $d(y,f'(x_1))$ follows just as the case discussed for $A_0$ and so $A_1$ is also nonempty. We complete the proof by transfinite induction. 
\end{proof}

The following lemma also holds. 

 \begin{lemma}
Let $X$ be a metric space, $A\subseteq X$ nonempty and $Y$ a hyperconvex metric space. Then $\Phi_c (f)$ is externally hyperconvex in ${\mathcal N}(X,Y)$ for every $f\in {\mathcal N}(A,Y)$.
 \end{lemma}

\begin{proof}
This proof follows the same patterns as the one of Lemma \ref{exthyp}. Let $f\in {\mathcal N}(A,Y)$, $\{ f_\alpha\}_{\alpha\in{\mathcal A}}\subseteq {\mathcal N}(X,Y)$ and $(r_\alpha)\subseteq {\mathbb R}^+$ such that $d_\infty (f_\alpha,f_\beta)\le r_\alpha+r_\beta$ and ${\rm dist}(f_\alpha, \Phi_c(f))\le r_\alpha$ for all $\alpha,\beta \in \mathcal{A}$. We need to prove that
\[\left( \bigcap_{\alpha\in{\mathcal A}} B(f_\alpha,r_\alpha)\right) \bigcap \Phi_c (f)\neq \emptyset.\]
It is possible to construct an extension $f'$ of $f$ in the above intersection by transfinite induction and this suffices to prove the lemma. Now the set $A_0$ to consider is given by:
\[A_0=\left( \bigcap_{a\in{ A}} B(f'(a),d(a,x_0))\right)\bigcap \left(\bigcap_{y\in Y}B(y,r_y(f(A)))\right)\bigcap \left( \bigcap_{\alpha\in{\mathcal A}} B(f_\alpha(x_0),r_\alpha)\right).\]
We need to check the hyperconvexity condition for $d(f'(a),y)$ with $a\in A$ and $y\in Y$, and $d(f_\alpha(x_0),y)$ for $\alpha \in {\mathcal A}$ and $y\in Y$. The first case is trivial as $f'(a)=f(a)$. For the second one, we need to recall that ${\rm dist}(f_\alpha,\Phi_c(f))\le r_\alpha$ and so, for $\varepsilon >0$, there exists $z(=z_\alpha)\in {\rm cov}(f(A))$ such that $d(f_\alpha(x_0),z)\le r_\alpha+\varepsilon$. Therefore, for $y\in Y$ and $\alpha\in {\mathcal A}$,
$$
d(y,f_\alpha (x_0))\le d(y,z)+d(z,f_\alpha(x_0))\le r_y(f(A))+ r_\alpha+\varepsilon.
$$
The hyperconvexity condition finally follows because $\varepsilon$ is arbitrary.

The set $A_1$ in this case is given by 
\[A_1 = A'_1 \cap B(f'(x_0),d(x_0,x_1)),\]
where
$$
 A'_1=\left( \bigcap_{a\in{ A}} B(f'(a),d(a,x_1))\right)\bigcap \left(\bigcap_{y\in Y}B(y,r_y(f(A)))\right) \bigcap \left( \bigcap_{\alpha\in{\mathcal A}} B(f_\alpha (x_1),r_\alpha)\right). 
 $$

We only need to check intersections with balls centered at $y\in Y$. The cases $d(f'(a),y)$ and $d(y,f_\alpha(x_1))$ follow as above. The case that remains to check is $d(y,f'(x_0))$ which follows by construction and, in general, by the inductive hypothesis. 
\end{proof}

Finally, we can state the following result.

\begin{theorem}\label{Phic}
Let $X$ be a metric space, $A\subseteq X$ nonempty and $Y$ a hyperconvex metric space. Then there exists a nonexpansive mapping $\alpha_c\colon {\mathcal N}(A,Y)\to {\mathcal N}(X,Y)$ such that for all $g\in {\mathcal N}(A,Y)$, $\alpha_c (g)(a)=g(a)$ for every $a\in A$ and $\alpha_c (g)(X)\subseteq {\rm cov}(g(A))$. 
\end{theorem}

\begin{remark}
Again, by considering adequate modifications in the proofs, one can see that Theorem \ref{Phic} holds in fact for $\lambda (A,Y)$ and $\lambda$-Lipschitz extensions. 
\end{remark} 

\begin{remark}
Theorem 1 in \cite{Sin89} gives the same result as Theorem \ref{Kha00} but for multivalued mappings with admissible values instead of externally hyperconvex subsets. As far as the authors know, Theorem \ref{main5} may be the first application of Theorem \ref{Kha00} where the externally hyperconvex condition plays a substantial role. In fact, values of the mapping $\Phi$, which have been proved to be externally hyperconvex, need not be admissible. Consider, for instance, $X$ as the real interval $[0,2]$, $A\subseteq X$ as $[1,2]$ and $Y=[0,1]$. Define $f\in{\mathcal N}(A,Y)$ as the function constantly equal to $1$. Then the functions $g(x)=1$ for $x\in X$ and
$$
h(x) = \begin{cases} x &\mbox{if } x\in [0,1] \\ 
1 & \mbox{if } x\in [1,2] \end{cases} 
$$
are in $\Phi (f)$. Therefore, any ball in ${\mathcal N}(X,Y)$ containing $\Phi (f)$ must be of radius at least $1/2$ and, in particular, it must contain the function constantly equal to $3/4$ which is not in $\Phi (f)$.
 \end{remark}

\begin{remark}
In this section we approached the case $\Phi_c$ in a direct way and not going through the metric projection as in Section \ref{convexity}. In contrast to the case of CAT$(\kappa)$ spaces with $\kappa\le 0$ where metric projections on closed and convex subsets are singlevalued, projections onto admissible subsets of hyperconvex metric spaces are multivalued. However, as it was shown in \cite{Sin89} they admit a nonexpansive selection (it was later shown in \cite{KhaKirMar00} that the same holds for externally hyperconvex subsets). This problem was further studied and the interested reader can find more about it in \cite{EspKha01,Esp05}. 
\end{remark}

\section{Acknowledgements}
Rafa Esp\' inola was supported by DGES, Grant MTM2012-34847C02-01 and Junta de Andaluc\'ia, Grant FQM-127. Adriana Nicolae was supported by a grant of the Romanian Ministry of Education, CNCS - UEFISCDI, project number PN-II-RU-PD-2012-3-0152. Part of this work was carried out while Adriana Nicolae was visiting the University of Seville. She would like to thank the Department of Mathematical Analysis and the Institute of Mathematics of the University of Seville (IMUS).

\end{document}